\documentclass[10pt,twoside]{siamltex}

\usepackage{amsfonts,epsfig}
\usepackage{amssymb,a4}
\usepackage{amsmath}
\usepackage{pmat}

\newtheorem{example}[theorem]{Example}

\newcommand{\matR}{\mathbb{R}}

\newcommand{\matC}{\mathbb{C}}

\begin{document}

\title{
Structure-preserving Schur methods for computing square roots
of real skew-Hamiltonian matrices
\thanks{
The first author is supported by the
National Natural Science Foundations of China, N$^\text{o}.\,10771022$  and N$^\text{o}.\,10571012$;
Scientific Research Foundation for the Returned Overseas Chinese
Scholars, State Education Ministry. The research of the remaining authors is supported by
FEDER Funds through ``Programa Operacional Factores de Competitividade - COMPETE" and by Portuguese Funds through FCT - "Funda\c{c}\~{a}o para a Ci\^{e}ncia e a Tecnologia", within the Project PEst-CMAT/UI0013/2011.
}}

\author{Zhongyun Liu
      \thanks{School of Math., Changsha University of Science \& Technology, Hunan, 410076, China {(\tt liuzhongyun@263.net)}. Corresponding author.}
        \and
        Yulin Zhang \thanks{Centro de Matem\'{a}tica, Universidade do Minho, 4710-057 Braga, Portugal
        (\tt zhang@math.uminho.pt, caferrei@math.uminho.pt, r\_ralha@math.uminho.pt).}
        \and
        Carla Ferreira \footnotemark[3]
        \and
        Rui Ralha \footnotemark[3]
        }

\pagestyle{myheadings} \markboth{Z.\ Liu, Y.\ L.\ Zhang, C.\
Ferreira and R.\ Ralha}{Square roots of real skew-Hamiltonian matrices}

\maketitle

\begin{abstract}
Our contribution is two-folded. First, starting from the known fact that
every real skew-Hamiltonian matrix has a real Hamiltonian square root, we
give a complete characterization of the square roots of a real
skew-Hamiltonian matrix $W$. Second, we propose a structure-exploiting
method for computing square roots of $W$. Compared to the standard real
Schur method, which ignores the structure, our method requires significantly
less arithmetic.
\end{abstract}





\section{Introduction}
Given  $A\in \matC^{n\times n}$, a matrix $X$ for which $X^2=A$ is
called a \textsl{square root} of $A$. The matrix square root is a
useful theoretical and computational tool, one of the most commonly
occurring matrix functions. See \cite{Gant77,Hig97,Highbook1,
HigMMT05,Highbook,Lancaster}.

The theory behind the existence of matrix square roots is nontrivial and the feature which complicates this
theory is that in general not all the square roots of a matrix $A$ are functions of $A$. See \cite{CroL74, Lancaster}.

It is well known that certain matrix structures can be inherited by
the square root. For example, a symmetric positive (semi)definite
matrix has a unique symmetric positive (semi)definite square root \cite{HornJohn91}.
The square roots of a centrosymmetric matrix are also centrosymmetric \cite{Liuzhr07}. A nonsingular
$M$-matrix has exactly one $M$-matrix as a square root. For an
$H$-matrix with positive diagonal elements there exists one and
only one square root which is also an $H$-matrix with positive
diagonal elements \cite{LinL01}. The principal square root of a
centrosymmetric $H$-matrix with positive diagonal elements
is a unique centrosymmetric $H$-matrix with positive diagonal entries \cite%
{LiuCC05b}. Any real skew-Hamiltonian matrix has a real Hamiltonian square root
\cite{FasMMX99}. In this paper we characterize such square roots.

For general matrices, an attractive method which uses the Schur decomposition is described by Bj\"orck and Hammarling \cite{BjoH83}
but may require complex arithmetic. Higham \cite{Hig87} presented a modification of this method which
enables real arithmetic to be used throughout when computing a real square root of a real matrix.
This method has been extended to compute matrix $p$th roots \cite{Smith03} and general matrix functions \cite{DavHig03}.

It is a basic tenet in numerical analysis that structure should be
exploited allowing, in general, the development of faster and/or
more accurate algorithms \cite{BBMehrmann93, MMT061}. We propose a
structure-exploiting method for computing square roots of a real
skew-Hamiltonian matrix $W$ which uses the real skew-Hamiltonian
Schur decomposition and requires significantly less arithmetic.

We give some basic definitions and establish notation in Section 2.
A description of the real Schur method and some results concerning
the existence of real square roots are presented in Section 3. In
Section \ref{sectionsqrtskew} we characterize the square roots of a
nonsingular  $W$ in a manner which makes clear the distinction
between the square roots which are functions of $W$ and those which
are not.  In Section \ref{AlgorithmSkew} we present our algorithms
for the computation of skew-Hamiltonian and Hamiltonian square
roots. In Section \ref{sectionExamples} we give results of numerical
experiments.

\section{Definitions and preliminaries results}

\subsection{Square roots of a nonsingular matrix}
Given a scalar function $f$ and a matrix $A\in \matC^{n\times n}$ there are many different ways
to define $f(A)$, a matrix of the same dimension of $A$, providing a useful generalization of a function
of a scalar variable.

It is a standard result that any matrix $A\in \matC^{n\times n}$ can be expressed in the Jordan canonical form
\begin{align}
\label{JordanForm}
Z^{-1}AZ=J=\diag(J_1,J_2,\ldots,J_p),\\
J_k=J_k(\lambda_k)=\begin{bmatrix}
\lambda_k & 1 \\
          &\lambda_k & \ddots\\
          &          & \ddots & 1\\
          &          &        &\lambda_k
\end{bmatrix} \in \matC^{m_k\times m_k}
\label{JordanBlock}
\end{align}
where $Z$ is nonsingular and $m_1+m_2+\cdots+m_p=n$. The Jordan matrix $J$ is unique up to the ordering of the blocks $J_i$.

Denote by $\lambda_1,\ldots,\lambda_s$ the $s$ distinct eigenvalues of $A$ and let $n_i$ be the order of the largest Jordan block in which $\lambda_i$
appears. The function $f$ is said to be \textsl{defined on the spectrum of} $A$ if the values
\[
f^{(j)}(\lambda_i), \quad j=0,\ldots,n_i-1, \quad i=1,\ldots,s,
\]
exist. These are called the \textsl{values of the function $f$ on the spectrum of $A$}.

The following definition of matrix function defines $f(A)$ to be a polynomial in the matrix $A$ completely determined by the values of $f$ on the spectrum of $A$. See \cite[p. 407 ff.]{Hig87}.

\begin{definition}
[matrix function via Hermite interpolation]
\label{DefFunM}
{\rm Let $f$ be defined on the spectrum of $A\in \matC^{n\times n}$. Then
$$
f(A):=p(A)
$$
where $p$ is the polynomial of degree less than $\sum_1^s n_i$ which satisfies the interpolation conditions
\[
p^{(j)}(\lambda_i)=f^{(j)}(\lambda_i),\qquad j=0,\ldots,n_i-1, \quad i=1,\ldots,s.
\]
There is a unique such $p$ and it is known as the Hermite interpolating polynomial.
}
\end{definition}

Of particular interest here is the function $f(z)=z^{1/2}$ which is certainly defined on the spectrum of $A$ if $A$ is nonsingular.
However, the square root function of $A$, $f(A)$, is not uniquely defined until one specifies which branch of the square root is to be taken in the neighborhood of each eigenvalue $\lambda_i$. Indeed, Definition \ref{DefFunM} yields a total of $2^s$ matrices $f(A)$ when all combinations of branches for the square roots $f(\lambda_i)$, $i=1,\ldots,s$, are taken. It is natural to ask whether these matrices are in fact square roots of $A$, that is, do we have $f(A)f(A)=A$? Indeed, these matrices, which are polynomials in $A$ by definition, are square roots of $A$.
See \cite{Highbook1,Lancaster}. However, these square roots are not necessarily all the square roots of $A$.

To classify all the square roots of a nonsingular matrix $A\in \matC^{n \times n}$ we need the following result concerning the square roots of a Jordan block.

\begin{lemma}
\label{Lemma1}
For $\lambda_k\neq 0$ the Jordan block $J_k(\lambda_k)$ in {\rm (\ref{JordanBlock})} has precisely two upper triangular square roots
\begin{equation} \label{JordanBlockSqr}
L_k^{(j)}=L_k^{(j)}(\lambda_k)=
\begin{bmatrix}
f(\lambda_k) & f^\prime(\lambda_k) &  \cdot &\ldots & \frac{f^{(m_k-1)}(\lambda_k)}{(m_k-1)!}\\
             & f(\lambda_k) & f^\prime(\lambda_k) & \ldots & \frac{f^{(m_k-2)}(\lambda_k)}{(m_k-2)!}\\
             & & \ddots &\ddots &\vdots\\
             & & & f(\lambda_k) & f^\prime(\lambda_k)\\
             & & &  &f(\lambda_k)
\end{bmatrix}, \quad j=1,2,
\end{equation}
where $f(\lambda)=\lambda^{1/2}$ and the superscript $j$ denotes the branch of the square root in the neighborhood of $\lambda_k$.
Both square roots are functions of $J_k$.
\end{lemma}

We will restrict our attention to matrices with real entries
and to investigate the real square roots of a real matrix we need to understand the structure of a general complex square root.
The following results allow us to obtain a useful characterisation of the square roots of a nonsingular matrix $A \in \matC^{n \times n}$ which are functions of $A$. See \cite{Gant77,Hig87}.

\begin{theorem} \label{Theorem01}
Let $A \in \matC^{n \times n}$ be nonsingular and have the Jordan canonical form {\rm (\ref{JordanBlock}).} Then all square roots $X$ of $A$ are given by
\[
X=ZU\diag\left(L_1^{(j_1)},L_2^{(j_2)},\ldots,L_p^{(j_p)}\right)U^{-1}Z^{-1},
\]
where $j_k=1$  or $j_k=2$ and $U$ is an arbitrary nonsingular matrix which commutes with $J$.
\end{theorem}

The following result extends Theorem \ref{Theorem01}.

\begin{theorem}\label{Theorem02}
Let the nonsingular matrix $A \in \matC^{n \times n}$ have the Jordan canonical form {\rm (\ref{JordanBlock})}
and let $s \leq p$ be the number of distinct eigenvalues of $A$. Then $A$ has precisely $2^s$ square roots which are functions of $A$, given by
\begin{equation}
X_j=Z\diag\left(L_1^{(j_1)},L_2^{(j_2)},\ldots,L_p^{(j_p)}\right)Z^{-1}, \quad j=1,\ldots,2^s,
\label{Equa1}
\end{equation}
corresponding to all possible choices of $j_1,\ldots,j_p$, $j_k=1$ or $j_k=2$, subject to the constraint that $j_i=j_k$ whenever $\lambda_i=\lambda_k$.

If $s<p$, $A$ has square roots which are not functions of $A$; they form parametrized families
\begin{equation}
X_j(U)=ZU\diag\left(L_1^{(j_1)},L_2^{(j_2)},\ldots,L_p^{(j_p)}\right)U^{-1}Z^{-1},\qquad j=2^s+1,\ldots,2^p,
\label{Equa2}
\end{equation}
where $j_k=1$  or $j_k=2$, $U$ is an arbitrary nonsingular matrix which commutes with $J$, and for each $j$ there exist
$i$ and $k$, depending on $j$, such that $\lambda_i=\lambda_k$ while $j_i\neq j_k$.
\end{theorem}

Proofs of these theorems and a description of the structure of the matrix $U$ can be found in \cite{Gant77}.
Note that formula in (\ref{Equa1}) follows from the fact that all square roots of $A$ which are functions of $A$
have the form
\[
f(A)=f(ZJZ^{-1})=Zf(J)Z^{-1}=Z\diag\big(f(J_k)\big)Z^{-1},
\]
and from Lemma \ref{Lemma1}.
The constrain on the branches $\{j_i\}$ follow from Definition \ref{DefFunM}.
The remaining square roots of $A$ (if any), which cannot be functions of $A$, are given by (\ref{Equa2}).

Theorem \ref{Theorem02} shows that the square roots of $A$ which are functions of $A$ are ``isolated" square roots, characterized by the fact that
the sum of any two of their eigenvalues is nonzero. On the other hand, the square roots which are not functions of $A$ form a finite number of
parametrized families of matrices: each family contains infinitely many square roots which share the same spectrum.

Some interesting corollaries follow directly from Theorem \ref{Theorem02}.

\begin{corollary}
If $\lambda_k\neq 0$, the two square roots of $J_k(\lambda_k)$ given in  {\rm Lemma \ref{Lemma1}} are the only square roots of
$J_k(\lambda_k)$.
\end{corollary}

\begin{corollary}\label{Corolary1}
If $A \in \matC^{n\times n}$ is nonsingular and in its Jordan canonical form {\rm (\ref{JordanBlock})} each eigenvalue appears
in only one Jordan block, then $A$ has precisely $2^p$ square roots, each of which is a function of $A$.
\end{corollary}

The final corollary is well known.
\begin{corollary}
Every Hermitian positive definite matrix has a unique Hermitian positive definite square root.
\end{corollary}

\subsection{Hamiltonian and skew-Hamiltonian matrices}
\label{HamiltSkew}
Hamiltonian and skew-Hamiltonian matrices have properties that follow directly from the definition.

\begin{definition}
\label{Defn1}
{\rm
Let $J=\left [%
\begin{array}{cc}
0 & I \\
-I & 0%
\end{array}%
\right ]$, where $I$ is the identity matrix of order $n$.

\begin{enumerate}
\item[(1)] A matrix $H\in\mathbb{R}^{2n\times 2n}$ is said to be
\emph{Hamiltonian} if $HJ=(HJ)^T$.\break Equivalently,
$H$ can be partitioned as
\begin{equation}
H=\begin{bmatrix}
A & G \\
F & -A^T%
\end{bmatrix}, \quad G=G^T, \quad F=F^T, \quad A, G, F \in \matR^{n\times n}.
\end{equation}

\item[(2)] A matrix $W\in\mathbb{R}^{2n\times 2n}$ is said to be \emph{%
skew-Hamiltonian} if \, $WJ=-(WJ)^T$.
Likewise, $W$ can be partitioned as
\begin{equation}  \label{skewh}
W=\left [%
\begin{array}{cc}
A & G \\
F & A^T%
\end{array}%
\right ], \quad G=-G^T, \quad F=-F^T,\quad A, G, F \in \matR^{n\times n}.
\end{equation}
\end{enumerate}
}
\end{definition}

These matrix structures induce particular spectral properties for $H$ and $W$.
Notably, the eigenvalues of $H$ are symmetric with respect to the imaginary axis and the eigenvalues
of $W$ have even algebraic and geometric multiplicities.

\begin{definition}$\,$
{\rm
\begin{enumerate}
\item[(1)] A matrix $S\in\mathbb{R}^{2n\times 2n}$ is said to be \emph{%
symplectic} if $SJ S^T= J$.

\item[(2)] A matrix $U\in\mathbb{R}^{2n\times 2n}$ is said to be \emph{%
orthogonal-symplectic} if $U$ is\break orthogonal and
symplectic. Any matrix belonging to this group can be partitioned as
\[
U=\begin{bmatrix}
U_{1} & U_{2} \\
-U_{2} & U_{1}%
\end{bmatrix}
\]
where $U_{i}\in \mathbb{R}^{n\times n},\, i=1,2.$
\end{enumerate}
}
\end{definition}

Hamiltonian and skew-Hamiltonian structures are preserved if symplectic similarity
transformations are used; if H is Hamiltonian (skew-Hamiltonian) and $S$ is symplectic, then
$S^{-1}HS$ is also Hamiltonian (skew-Hamiltonian).
In the interest of numerical stability the similarities should be orthogonal as well.


The first simplifying reduction of a skew-Hamiltonian matrix was introduced
by Van Loan in \cite{Loan84}. But first we recall the real Schur decomposition \cite{GolVL96}.

\begin{theorem}
[real Schur form]
\label{Theorem05}
If $A\in \matR^{n\times n}$, then there exists a real orthogonal matrix $Q$ such that
\begin{equation}
Q^TAQ=R=\begin{bmatrix}
R_{11} & R_{12} & \ldots & R_{1m}\\
       & R_{22} & \ldots & R_{2m}\\
       &        & \ddots & \vdots\\
       &        &        & R_{mm}
\end{bmatrix}\in \matR^{n\times n},
\label{RealSchur}
\end{equation}
where each block $R_{ii}$ is either $1\times 1$ or $2 \times 2$ with complex conjugate eigenvalues
$\lambda_i$ and $\bar{\lambda}_i$, $\lambda_i\neq \bar{\lambda}_i$ ($R$ is in quasi-upper triangular form).
\end{theorem}

In \cite{Loan84} it was shown that any skew-Hamiltonian $W$ can be brought
to block-upper-triangular form by an orthogonal-symplectic similarity. Actually,
we can explicitly compute an orthogonal-symplectic matrix $U$ such that
\[
U^TWU=\begin{bmatrix}
W_1 & W_2 \\
O & W_1^T
\end{bmatrix},
\]
where $W_2^T=-W_2$ and $W_1$ is upper Hessenberg
(a matrix is \textsl{upper Hessenberg} if all entries  below its first subdiagonal
are zero). This is called the \textsl{symplectic Paige/Van Loan (PVL) form}.

Subsequently, if the standard QR algorithm is applied to $W_1$ producing
an orthogonal matrix $Q$ and a matrix in real Schur form $N_1$ so that
\[
W_1=QN_1Q^T,
\]
we attain the \textsl{real skew-Hamiltonian Schur decomposition} of $W$,
\begin{equation}
{\cal U}^TW{\cal U}=\begin{bmatrix}
N_1 & N_2 \\
O & N_1^T
\end{bmatrix},
\end{equation}
where
$
{\cal U}=U \begin{bmatrix} Q & O\\
O & Q
\end{bmatrix}
$
and $N_2=Q^TW_2Q$.

\begin{lemma}[real skew-Hamiltonian Schur form]
\label{lemma2}
Let $W \in \matR^{2n\times 2n}$ be skew-Hamiltonian. Then there exists an orthogonal
matrix
\[
{\cal U}=\begin{bmatrix}
{\cal U}_{1} & {\cal U}_{2} \\
-{\cal U}_{2} & {\cal U}_{1}%
\end{bmatrix}, \quad {\cal U}_{1},{\cal U}_{2} \in \mathbb{R}^{n\times n},
\]
such that
\begin{equation}
\label{rshshur}
\mathcal{U}^TW\mathcal{U}=
\begin{bmatrix}
N_{1} & N_{2} \\
0 & N_{1}^T
\end{bmatrix}, \quad N_{2}^T=-N_{2},
\end{equation}
and $N_{1}$ is in real Schur form.
\end{lemma}

In \cite[Theorem 5.1]{PaigeLoan81} we can find a result concerning the \textsl{real Hamiltonian Schur decomposition.}

\begin{theorem}[real Hamiltonian Schur form]
\label{Defn2}
Let $H \in \matR^{2n\times 2n}$ be Hamiltonian. If $H$ has no nonzero purely imaginary eigenvalues, then there exists an orthogonal
matrix
\[
{\cal U}=\begin{bmatrix}
{\cal U}_{1} & {\cal U}_{2} \\
-{\cal U}_{2} & {\cal U}_{1}%
\end{bmatrix}, \quad {\cal U}_{1},{\cal U}_{2} \in \mathbb{R}^{n\times n},
\]
such that
\begin{equation}
\label{rhshur}
\mathcal{U}^TH\mathcal{U}=
\begin{bmatrix}
H_{1} & H_{2} \\
0 & -H_{1}^T%
\end{bmatrix}, \quad H_{2}^T=H_{2},
\end{equation}
and $H_{1}$ is in real Schur form.
\end{theorem}

In this article we are interested in the computation of a real square root of a real skew-Hamiltonian matrix $W$ and
to discuss square roots of a skew-Hamiltonian matrix we need to consider a variant of the Jordan canonical form (\ref{JordanBlock}) when the matrix
$A$ is real. In this case, all the nonreal eigenvalues must occur in conjugate pairs and
all the Jordan blocks of all sizes (not just $1\times 1$ blocks) corresponding to nonreal eigenvalues occur
in conjugate pairs of equal size.

For example, if $\lambda$ is a nonreal eigenvalue of the real matrix $A$, and if $J_2(\lambda)$
appears in the Jordan canonical form of $A$ with a certain multiplicity, $J_2(\bar{\lambda})$ must also appear
with the same multiplicity. See \cite[p.150 ff.]{HornJohn85}. The block matrix
\begin{equation}\label{blockPair}
\begin{bmatrix}
J_2(\lambda) & O\\
O & J_2(\bar{\lambda})
\end{bmatrix}=
\begin{pmat}[{.|..}]
\lambda & 1\,       & \, 0 & 0 \cr
      0 &\lambda  & 0 & 0 \cr\-
      0 & 0       & \bar{\lambda} & 1 \cr
      0 & 0       & 0 & \bar{\lambda} \cr
\end{pmat}
\end{equation}
is permutation-similar (interchange rows and columns 2 and 3) to the block matrix
\begin{equation}
\label{Dlambda}
\begin{pmat}[{.|..}]
       \lambda & 0\,       & 1 & 0 \cr
             0 & \bar{\lambda}  & 0 & 1 \cr\-
      0 & 0       & \lambda & 0 \cr
      0 & 0       & 0 & \bar{\lambda} \cr
\end{pmat} =
\begin{bmatrix}
D(\lambda) & I\\
0 & D(\lambda)
\end{bmatrix},
\qquad D(\lambda):=\begin{bmatrix}
\lambda & 0\\
O & \bar{\lambda}
\end{bmatrix}.
\end{equation}
Each block $D(\lambda)$ is similar to the matrix
\begin{equation}\label{C(a,b)}
SD(\lambda)S^{-1}=\begin{bmatrix}
a & b\\
-b & a
\end{bmatrix}:=C(a,b), \qquad S=\begin{bmatrix} -i & -i \\ 1 &-1\end{bmatrix},
\end{equation}
where $\lambda, \bar{\lambda}=a\pm ib$, $a,b\in \matR, \, b \neq 0 $. Thus, every block pair of conjugate $2\times 2$
Jordan blocks (\ref{blockPair}) with nonreal eigenvalue $\lambda$ is similar to a real $4\times 4$ block of the form
\[
\begin{pmat}[{.|..}]
       a & b\,       & 1 & 0 \cr
      -b & a  & 0 & 1 \cr\-
      0 & 0       & a & b \cr
      0 & 0       & -b & a \cr
\end{pmat} =
\begin{bmatrix}
C(a,b) & I\\
O & C(a,b)
\end{bmatrix}.
\]
In general, every block pair of conjugate $k\times k$ Jordan blocks
with nonreal $\lambda$,
\begin{equation}\label{2Jordan}
\begin{bmatrix}
J_k(\lambda) & O\\
O & J_k(\bar{\lambda})
\end{bmatrix},
\end{equation}
is similar to a real $2k\times 2k$ block matrix of the form
\begin{equation}  \label{rjordanb}
C_k(a,b)=\begin{bmatrix}%
C(a,b) & I &  &  &  \\
& C(a,b) & I &  &  \\
&  & \ddots & \ddots &  \\
&  &  & C(a,b) & I \\
&  &  &  & C(a,b)%
\end{bmatrix}.%
\end{equation}
We call $C_k(a,b)$ a \textsl{real Jordan block.} These observations lead us to the \textsl{real Jordan canonical form}.


\begin{theorem}
Each matrix $A\in \matR^{n\times n}$ is similar (via a \emph{real} similarity transformation)
to a block diagonal real matrix of the form
\begin{equation}  \label{rjordanm}
J_R={\begin{bmatrix}
C_{n_1}(a_1, b_1) &  &  &  &  &  \\
& \ddots &  &  &  &  \\
&  & C_{n_p}(a_p,b_p) &  &  &  \\
&  &  & J_{n_{p+1}}(\lambda_{p+1}) &  &  \\
&  &  &  & \ddots &  \\
&  &  &  &  & J_{n_{p+q}}(\lambda_{p+q})%
\end{bmatrix}},
\end{equation}
where $\lambda_k=a_k+ib_k$, $a_k,b_k \in \matR$, $k=1,\ldots,p$,
is a nonreal eigenvalue of $A$ and $\lambda_k$, $k=p+1,\ldots,p+q$, is a real eigenvalue of $A$.
Each real Jordan block $C_{n_k}(a_k,b_k)$
is of the form {\rm (\ref{rjordanb})} and corresponds to a pair of conjugate Jordan blocks
$J_{n_k}(\lambda_k)$ and $J_{n_k}(\bar{\lambda}_k)$ for a nonreal $\lambda_k$
in the Jordan canonical form of $A$ {\rm (\ref{JordanForm})}.
The real Jordan blocks $J_{n_k}(\lambda_k)$ are exactly the Jordan blocks in
{\rm(\ref{JordanForm})} with real $\lambda_k$. Notice that $2(n_1+\cdots+n_p)+(n_{p+1}+\cdots+n_{p+q})=n$.
We call $J_R$ a \emph{real Jordan matrix} of order $n$, a direct sum of real Jordan blocks.


\end{theorem}

In \cite{FasMMX99} it is shown that every real
skew-Hamiltonian matrix can also be reduced to a \textsl{real skew-Hamiltonian Jordan
form} via a symplectic similarity. See also \cite{FassIkara05}.

\begin{lemma}
\label{lemma3}\cite[Theorem 1]{FasMMX99} For every real skew-Hamiltonian
matrix $W\in \matR^{2n\times 2n}$ there exists a symplectic matrix ${\Psi}\in \matR^{2n\times 2n}$ such that
\begin{equation}  \label{rjordan}
{\Psi}^{-1}W{\Psi}=\left [%
\begin{array}{cc}
J_R &  \\
& J_R^T%
\end{array}%
\right ],
\end{equation}
where $J_R\in \mathbb{R}^{n\times n}$ is in real Jordan form
{\rm (\ref{rjordanm})} and is unique up to a permutation of real Jordan
blocks.
\end{lemma}

In principle, a real square root of $W$ can be obtained by the general method devised by Higham
\cite{Hig87}. Such method, however, does not exploit the structure of $W$. The method we propose
exploit the skew-Hamiltonian structure of $W$ and it also uses the real Schur method.
To make our presentation simpler we decided to present the details of the real Schur method.
See \cite[p. 412 ff.]{Hig87}.


\section{An algorithm for computing real square roots}
\subsection{The Schur method}

Bj\"orck and Hammarling \cite{BjoH83} presented an excellent method for computing a square
root of a matrix $A$. Their method first computes a Schur decomposition
\[
Q^*AQ=T
\]
where $Q$ is unitary and $T$ is upper triangular \cite{GolVL96}, and then determines an upper triangular
square root $Y$ of $T$ with the aid of a fast recursion. A square root of $A$ is given by
\[
X=QYQ^*.
\]
A disadvantage of this Schur method is that if $A$ is real and has nonreal eigenvalues, the method needs
complex arithmetic even if the square root which is computed should be real. When computing a real square root
it is obviously desirable to work with real arithmetic; depending on the relative costs of real and complex arithmetic
on a given computer system, substantial computational savings may occur, and moreover, a computed real square root is guaranteed.

Higham described a generalization of the Schur method which enables real arithmetic to be used
throughout when computing a real square root of a real matrix. In Section \ref{SectionRealShur} we present this method.
First we give, for a nonsingular matrix  $A\in \matR^{n\times n}$,
conditions for the existence of a real square root, and for the existence of a real square root which is a polynomial in $A$.

\subsection{Existence of real square roots}
\label{SectionExistence}
The following result concerns the existence of general real square roots - those which are
not necessarily functions of $A$.

\begin{theorem}
\label{Theorem03}
Let $A\in \matR^{n\times n}$ be nonsingular. $A$ has a real square root if and only if each elementary divisor
of $A$ corresponding to a real negative eigenvalue occurs an even number of times.
\end{theorem}

Theorem \ref{Theorem03} is mainly of theoretical interest, since the proof is nonconstructive
and the condition for the existence of a real square root is not easily checked computationally.
We now focus attention on the real square roots of $A\in \matR^{n\times n}$. The key to analysing the
existence of square roots of this type is the real Schur decomposition.

Suppose that $A\in \matR^{n\times n}$ and that $f$ is defined on the spectrum of $A$ and
consider the real Shcur form of $A$ in (\ref{RealSchur}). Since $A$ and $R$
in (\ref{RealSchur}) are similar, we have
\[
f(A)=Qf(R)Q^T,
\]
so that $f(A)$ is real if and only if
\[
Z=f(R)
\]
is real. It is not difficult to show that $Z$ inherits $R's$ quasi-upper triangular structure and that
\[
Z_{ii}=f(R_{ii}), \qquad i=1,\ldots,m.
\]

If $A$ is nonsingular and $f$ is the square root function, then we have
\[
Z^2=R \quad \text{and} \quad X^2=A \quad \text{with} \quad X=QZQ^T.
\]
The whole of $Z$ is uniquely determined by its diagonal
blocks. To see this equate $(i,j)$ blocks in the equation $Z^2=R$
to obtain
\[
\sum_{k=i}^j Z_{ik}Z_{kj}=R_{ij}, \quad j\ge i.
\]
These equations can be recast in the form
\begin{gather}
\label{equation1}
Z_{ii}^2=R_{ii}, \quad i=1,\ldots,m\\
\label{equation2}
Z_{ii}Z_{ij}+Z_{ij}Z_{jj}=R_{ij}-\sum_{k=i+1}^{j-1}Z_{ik}Z_{kj}, \quad j=i+1,\ldots,m.
\end{gather}
Thus, if the diagonal blocks $Z_{ii}$ are known, (\ref{equation1}) provides an algorithm for computing the remaining
blocks $Z_{ij}$ of $Z$ along one superdiagonal at a time in the order specified by $j-i=1,2,\ldots,m-1$. The condition for
the Sylvester equation $(\ref{equation2})$ to have a unique solution $Z_{ij}$ is that $Z_{ii}$ and $-Z_{jj}$ have no
eigenvalue in common \cite{GolVL96,Lancaster}. This is guaranteed because the eigenvalues of $Z$ are
$\mu_k=f(\lambda_k)$ and for the square root function $f(\lambda_i)=-f(\lambda_j)$ implies that
$\lambda_i=\lambda_j=0$, contradicting the nonsingularity of $A$.

From this algorithm for constructing $Z$ from its diagonal blocks we conclude that $Z$ is real
and hence $f(A)$ is real, if and only if each of the blocks $Z_{ii}=f(R_{ii})$ is real.

We now examine the square roots $f(A)$ of a $2\times 2$ matrix $A$ with complex conjugate eigenvalues.
Since $A$ has 2 distinct
eigenvalues, it follows from\break Corollary \ref{Corolary1} that $A$ has four square roots which are all functions of $A$.
Next lemma says about the form os these square roots.
\begin{lemma}
\label{Lemma7}
Let $A\in \matR^{2\times 2}$ have complex eigenvalues $\lambda, \bar{\lambda}=\theta\pm i \mu$,
where $\mu \neq 0$. Then $A$ has four square roots, each of which is a function of $A$. Two of the square
roots are real, with complex conjugate eigenvalues, and two are pure imaginary, having eigenvalues which are
not complex conjugate.
\end{lemma}

Using this lemma it can be proved

\begin{theorem}
\label{Theorem04}
Let $A\in \matR^{n\times n}$ be nonsingular. If $A$ has a real negative eigenvalue, then $A$ has
no real square roots which are functions of $A$.

If $A$ has no real negative eigenvalues, then there are precisely $2^{r+c}$ real square roots of $A$
which are functions of $A$, where $r$ is the number of distinct real eigenvalues of $A$ and $c$ is the
number of distinct complex conjugate eigenvalue pairs.
\end{theorem}

For proofs of Lemma \ref{Lemma7} and Theorem \ref{Theorem04} see \cite[p. 414 ff.]{Hig87}.

It is clear from Theorem \ref{Theorem03} that $A$ may have real negative eigenvalues and yet still have
a real square root; however, as Theorem \ref{Theorem04} shows, the square root will not be a function of $A$.

\subsection{The real Schur method}
\label{SectionRealShur}
The ideas of the last section lead to a natural extension of Bj\"orck and Hammarling's Schur
method for computing in real arithmetic a real square root of a nonsingular $A\in \matR^{n\times n}$.
This real Schur method begins by computing a real Schur decomposition (\ref{RealSchur}), then computes
a square root $Z$ of $R$ from equations (\ref{equation1}) and (\ref{equation2}), and finally obtains a square
root of $A$ via the transformation $X=QZQ^T$.

The solution of Equation (\ref{equation1}) can be computed efficiently in
a way suggested by the proof of Lemma \ref{Lemma7} \cite[p. 417]{Hig87}. The first step is to compute $\theta$
and $\mu$, where $\lambda=\theta + i\mu$ is an eigenvalue of the matrix
\[
R_{ii}=\begin{bmatrix}
r_{11} & r_{12}\\
r_{21} & r_{22}
\end{bmatrix}.
\]
Next, $\alpha$ and $\beta$ such that $(\alpha+i\beta)^2=\theta + i\mu$ are required. Finally, the
real square roots of $R_{ii}$ are given by
\begin{equation}\label{Zii}
Z_{ii}=\pm\left(\alpha I + \frac{1}{2\alpha}\big(R_{ii}-\theta I\big)\right).
\end{equation}

If $Z_{ii}$ is of order $p$ and $Z_{jj}$ is of order $q$, Equation (\ref{equation2}) can be written as
\begin{equation}
\left(I_q\otimes Z_{ii}+Z_{jj}^T \otimes I_p\right)\text{col}(Z_{ij})=\text{col}\left(R_{ij}-\sum_{k=i+1}^{j-1}Z_{ik}Z_{kj}\right),
\quad j>i
\label{linearSystem}
\end{equation}
where $\otimes$ is the Kronecker product and $\text{col}(M)$ denotes the column vector formed by taking columns of $M$
and stacking them atop one another from left to right. The linear system (\ref{linearSystem}) is of order $pq=1,2$ or 4 and
may be solved by standard methods.

Note that to conform with the definition of $f(A)$ we have to choose the signs in $(\ref{Zii})$ so that $Z_{ii}$ and $Z_{jj}$
have the same eigenvalues whenever $R_{ii}$ and $R_{jj}$ do; this choice ensures simultaneously the nonsingularity
of the linear systems (\ref{linearSystem}).

Any of the real square roots $f(A)$ of $A$ can be computed in the above fashion by the real Schur method.

\label{pageRealSchur}
\medskip

\textbf{Algorithm 1 [Real Schur method]}

\medskip

\begin{enumerate}
\item[1.] compute a real Schur decomposition of $A$,
$$
A=Q^TRQ;
$$
\item[2.] compute a square root $Z$ of $R$ solving the equation $Z^2=R$ via
\begin{gather*}
Z_{ii}^2=R_{ii}, \qquad  1 \le i \le m,\\
\left(I_q\otimes Z_{ii}+Z_{jj}^T \otimes I_p\right)\text{col}(Z_{ij})=\text{col}\left(R_{ij}-\sum_{k=i+1}^{j-1}Z_{ik}Z_{kj}\right), \quad j>i
\end{gather*}
[block fast recursion]
\item[3.] obtain a square root of $A$,  $X=QZQ^T$.
\end{enumerate}

\pagebreak

The cost of the real Schur method, measured in floating point
operations (flops) may be broken down as follows. The real Schur
factorization costs about $15n^3$ flops \cite{GolVL96}. The
computation of $Z$ requires $n^3/6$ flops and the formation of
$X=QZQ^T$ requires $3n^3/2$ flops \cite[p. 418]{Hig87}. Only a
fraction of the overall time is spent in computing the square root
$Z$.

\section{Square roots of a skew-Hamiltonian matrix}
\label{sectionsqrtskew}


In this section we present a detailed classification of the square
roots of a skew-Hamiltonian matrix $W \in \matR^{2n\times 2n}$ based on its real skew-Hamiltonian Jordan form (\ref{rjordan}),
\begin{equation}
\label{skewJordan}
{\Psi}^{-1}W{\Psi}=\left [%
\begin{array}{cc}
J_R &  \\
& J_R^T%
\end{array}%
\right ],
\end{equation}
where $J_R\in \mathbb{R}^{n\times n}$ is in real Jordan form
(\ref{rjordanm}) and $\Psi$ is a sympletic matrix. First, we will discuss the square roots of $J_R$.

According to Lemma \ref{Lemma1}, for $\lambda_k\neq 0$ a canonical Jordan block $J_k(\lambda_k)$ has precisely two upper triangular square roots
given by (\ref{JordanBlockSqr}). As a corollary it follows that

\begin{corollary}
\label{coro}  For a real eigenvalue $\lambda_k\neq 0$, the Jordan block $J_k(\lambda_k)$ in {\rm (\ref{JordanBlock})} has precisely
two upper triangular square roots which are real, if $\lambda >0$, and pure imaginary, if $\lambda<0$. Both square roots are functions
of $J_k$.
\end{corollary}

To fully characterize the square roots of a real Jordan block $C_k(a,b)$, we first examine
the square roots of a $2\times 2$ block $C(a, b)$ corresponding to the nonreal eigenvalues
$\lambda, \bar{\lambda}=a\pm ib$, $a,b\in \matR$. According to Lemma \ref{Lemma7},
the real Jordan block $C(a,b)$ in (\ref{C(a,b)}) has four square roots, each of which is a function
of $C(a,b)$. Two of the square roots are real and two are pure imaginary.


\begin{lemma}
\label{lemma6} The real Jordan block $C_k(a, b)$ in {\rm (\ref{rjordanb})}
has precisely four block upper triangular square roots
\begin{equation}
F_k^{(j)}=\begin{bmatrix}
F & F_1 & \cdot & \ldots & F_{k-1} \\
& F & F_1 & \ldots & F_{k-2} \\
&  & \ddots & \ddots & \vdots \\
&  &  & F & F_1 \\
&  &  &  & F%
\end{bmatrix},\quad j=1,\ldots,4,
\end{equation}
where $F$ is a square root of $C(a,b)$ and $F_{i}$, $i=1,\ldots,k-1$, are the unique solutions
of certain Sylvester equations. The superscript $j$ denotes one of the four square roots of $C(a,b)$.
These four square roots $F_k^{(j)}$ are functions of $C_k(a,b)$, two of them are real and two are pure imaginary.
\end{lemma}

\begin{proof}
Since $C_k(a,b)$ has 2 distinct eigenvalues and the Jordan form
(\ref{2Jordan}) has $p=2$ blocks, from Corollary \ref{Corolary1} we
know that $C_k(a,b)$ has four square roots which are all functions
of $A$.

Let $X$ be a square root of $C_k(a,b)$ ($k>1$).
It is not difficult to see that $X$ inherits $C_k(a,b)$ block upper triangular structure,
\begin{equation}
X=\begin{bmatrix}
X_{11} & X_{12} & \ldots & X_{1,k}\\
       & X_{22} & \ldots & X_{2,k}\\
       &        &\ddots  & \vdots\\
       &        &        & X_{kk}
\end{bmatrix}
\end{equation}
where $X_{i,j}$ are all $2\times 2$ matrices. Equating $(i,j)$ blocks in the equation
\[
X^2=C_k(a,b)
\]
we obtain
\begin{gather}
\label{equationA}
X_{ii}^2=C(a,b), \quad i=1,\ldots,k,\\
\label{equationB}
X_{ii}X_{i,i+1}+X_{i,i+1}X_{i+1,i+1}=I_2,\quad i=1,\ldots,k-1\\
\label{equationC}
X_{ii}X_{ij}+X_{ij}X_{jj}=-\sum_{l=i+1}^{j-1}X_{il}X_{lj}, \quad j=i+2,\ldots,k.
\end{gather}
The whole of $X$ is uniquely determined by its diagonal blocks. If $F$ is one square root of $C(a,b)$,
from (\ref{equationA}) and to conform with the definition of $f(A)$ (the eigenvalues of $X_{ii}$ must be the same),
we have
\begin{equation}
\label{equationD}
X_{ii}=F, \quad i=1,\ldots,k.
\end{equation}

Equations (\ref{equationB}) and (\ref{equationC}) are Sylvester equations and the condition for them
to have a unique solution $X_{ij}$ is that $X_{ii}$ and $-X_{jj}$ have no eigenvalues in common and this
is guaranteed.

From (\ref{equationB}) we obtain the blocks $X_{ij}$ along the first superdiagonal and (\ref{equationD}) forces
them to be all equal, say $F_1$,
\[
X_{12}=X_{23}=\ldots=X_{k-1,k}=F_1.
\]
This implies that the other superdiagonals obtained from (\ref{equationC}) are also constant, say $F_{j-1}$, $j=3,\ldots,k$,
\[
X_{1j}=X_{2,j+1}=\ldots=X_{k-j+1,k}=F_{j-1}, \quad j=3,\ldots,k.
\]

Thus, since there are only exactly four distint square roots of $C(a,b)$ which are functions of $C(a,b)$,
$F=F^{(l)}$, $l=1,\ldots,4$, it follows that $C_k(a,b)$ will also have precisely four square roots which are functions of $C_k(a,b)$. If $F$ is real then $F_{j-1}$, $j=2,\ldots,k$, will also be real.
If $F$ is pure imaginary it can also be seen that $F_{j-1}$, $j=2,\ldots,k$, will be pure imaginary too.
\end{proof}

Next theorem combines Corollary \ref{coro} and Lemma \ref{lemma6} to characterize the square roots of a real Jordan matrix
$J_R$.

\begin{theorem}
\label{theorem1}
Assume that a nonsingular real Jordan matrix $J_R$ in {\rm(\ref{rjordanm})} has $p$ real Jordan blocks
corresponding to $c$ distinct complex conjugate eigenvalue pairs and $q$ canonical Jordan blocks corresponding to
$r$ distinct real eigenvalues.

Then $J_R$ has precisely $2^{2c+{r}}$ square roots which are
functions of $J_R$, given by
\begin{equation}  \label{rmatrixsqrtf}
X_j=\diag\left(F_{n_1}^{(j_1)}, \ldots,
F_{n_{p}}^{(j_{p})}, L_{n_{p+1}}^{(i_1)}, \ldots,
L_{n_{p+q}}^{(i_{q})}\right), \quad j=1,\ldots,2^{2c+{r}},
\end{equation}
corresponding to all possible choices of $j_1, \ldots, j_{p}$, $j_k=1,2,3$ or 4,
and $i_1, \ldots, i_{q}$, $i_k=1$ or 2, subject
to the constraint that $j_l=j_k$ and $i_l=i_k$ whenever
$\lambda_l=\lambda_k$.

If $c+r < p+q$, then $J_R$ has square roots which
are not functions of $J_R$ and they form
$2^{2p+q}-2^{2c+{r}}$ parameterized families
given by
\begin{eqnarray}  \label{rmatrixsqrtnf}
X_j(\Omega)=\Omega \diag\left(F_{n_1}^{(j_1)}, \ldots, F_{n_{p%
}}^{(j_{p})}, L_{n_{p+1}}^{(i_1)}, \ldots, L_{n_{p+q}}^{(i_{q%
})}\right)\Omega^{-1}, \\
\qquad \qquad \qquad j= 2^{2c+{r}}+1, \ldots, 2^{2p+{q}},  \nonumber
\end{eqnarray}
where $j_k=1,2,3$ or 4 and $i_k=1$ or 2, $\Omega$ is an arbitrary
nonsingular matrix which commutes with $J_R$ and for each $j$ there
exist $l$ and $k$
depending on $j$, such that $\lambda_l=\lambda_k$ while $j_l\ne j_k$ or $%
i_l\ne i_k$.
\end{theorem}

\begin{proof}
The number of distinct eigenvalues is $s=2c+{r}$ and, according to Theorem \ref{Theorem02},
$J_R$ has precisely $2^s=2^{2c+{r}}$ square roots which are functions of $J_R$.
All square roots of $J_R$ which are functions of $J_R$ satisfy
\begin{align*}
f(J_R)&  = \left [%
\begin{array}{cccccc}
f(C_{n_1}) &  &  &  &  &  \\
& \ddots &  &  &  &  \\
&  & f(C_{n_{p}}) &  &  &  \\
&  &  & f(J_{n_{p+1}}) &  &  \\
&  &  &  & \ddots &  \\
&  &  &  &  & f(J_{n_{p+{q}}})%
\end{array}%
\right ]
\end{align*}
and, according to Lemma \ref{Lemma1} and Lemma \ref{lemma6} these are given by (\ref{rmatrixsqrtf}).
The constraint on the branches $\{j_k\}$ and $\{i_k\}$ comes from Definition \ref{DefFunM} of matrix
\break function. The remaining square roots of $J_R$, if they
exist, cannot be functions of $J_R$. Equation
(\ref{rmatrixsqrtnf}) derives from the second part of Theorem \ref{Theorem02}.
\end{proof}

From Theorem \ref{theorem1}, Lemma \ref{lemma6} and Corollary \ref{coro} the next result concerning the square roots
of $A$ which are functions of $A$ follows immediately.

\begin{corollary}
Under the assumptions of  {\rm Theorem \ref{theorem1}},

\begin{enumerate}
\item[(1)] if $J_R$ has a real negative eigenvalue, then $J_R$ has no real square roots which are functions of
$J_R$;
\item[(2)] if $J_R$ has no real negative eigenvalues, then $J_R$ has precisely $2^{c+{r}}$ real square roots
which are functions of $J_R$, given by {\rm(\ref{rmatrixsqrtf})} with the choices of
$j_1, \ldots, j_p$ corresponding to real square roots $F_{n_1}^{(j_1)}, \ldots,
F_{n_{p}}^{(j_{p})}$;
\item[(3)] if $J_R$ has no real positive eigenvalues, then $J_R$ has precisely $2^{c+{r}}$ pure\break imaginary
square roots which are functions of $J_R$, given by
{\rm(\ref{rmatrixsqrtf})} with the choices of $j_1, \ldots, j_p$
corresponding to pure imaginary square roots \break
$F_{n_1}^{(j_1)},\dots,F_{n_p}^{j_p}$.
\end{enumerate}

\end{corollary}

Now we want to use all these results to characterize the square roots of a real skew-Hamiltonian matrix $W$.

\begin{theorem}
\label{Theorem2} Let $W\in\matR^{2n\times 2n}$ be a nonsingular skew-Hamiltonian matrix with
the real skew-Hamiltonian Jordan form in {\rm(\ref{skewJordan})}. Assume that $J_R$ has $p$ real Jordan blocks
corresponding to $c$ distinct complex conjugate eigenvalue pairs and $q$ canonical Jordan blocks corresponding to
$r$ distinct real eigenvalues.

Then $W$ has precisely $2^{2c+{r}}$ square roots
which are functions of $W$, given by
\begin{equation}  \label{cshamilsqrt}
Y_j={\Psi} \diag\left(X_j, {X_j}^T\right){\Psi}^{-1},\quad j=1,\ldots,2^{2c+r},
\end{equation}
where $X_j$ is a square root of $J_R$ given in (\ref{rmatrixsqrtf}).

$W$ has always square roots
which are not functions of $W$ and they form
$4^{2p+q}-2^{2c+{r}}$ parameterized families given by
\begin{equation}  \label{fcshamilsqrt}
Y_j(\Theta)=\Psi\Theta \diag\left(\widetilde{X}_j, {\widehat{X}_j}^T\right)\Theta^{-1}\Psi^{-1},
 \quad j= 2^{2c+{r}}+1, \ldots, 4^{2p+{q}},
\end{equation}
where
\begin{align*}
\widetilde{X}_j&=\diag\left(F_{n_1}^{(j_1)}, \ldots, F_{n_{p%
}}^{(j_{p})}, L_{n_{p+1}}^{(i_1)}, \ldots, L_{n_{p+q}}^{(i_{q%
})}\right),\\
\widehat{X}_j&=\diag\left(F_{n_1}^{(j_{p+1})}, \ldots, F_{n_{p%
}}^{(j_{2p})}, L_{n_{p+q}}^{(i_{q+1})}, \ldots, L_{n_{p+q}}^{(i_{2q%
})}\right),
\end{align*}
$j_k=1,2,3$ or 4 and $i_k=1$ or 2,
$\Theta$ is an arbitrary nonsingular matrix which
commutes with  $\diag\left(J_R, J_R^T\right)$ and for each $j$ there
exist $l$ and $k$ depending on $j$, such that $\lambda_l=\lambda_k$ while $j_l\ne j_k$ or
$i_l\ne i_k$.
\end{theorem}

Notice that $W$ has $s=2c+r$ distinct eigenvalues corresponding to $2(p+q)$ real Jordan blocks and $2(2p+q)$ canonical Jordan blocks.
We always have $s\leq 2p+q$ and so $s < 2(2p+q)$. Thus, there are always square roots which are not functions of $W$.

\begin{proof}
This result is a direct consequence of Theorem \ref{Theorem02}
and Theorem \ref{theorem1}. Notice that if $X_j$ in
(\ref{rmatrixsqrtf}) is a square root of $J_R$ then
$\diag\left(X_j,X_j^T\right)$ is a square root of
$\diag\left(J_R,J_R^T\right)$. Thus,
\begin{align*}
W&=\Psi \begin{bmatrix}
J_R &\\
    &J_R^T
\end{bmatrix} \Psi^{-1}\\
&=
\Psi \begin{bmatrix}
X_j &\\
    &X_j^T
\end{bmatrix}
\begin{bmatrix}
X_j &\\
    &X_j^T
\end{bmatrix} \Psi^{-1}\\
&=
\Psi \begin{bmatrix}
X_j &\\
    &X_j^T
\end{bmatrix}
\Psi^{-1}
\Psi
\begin{bmatrix}
X_j &\\
    &X_j^T
\end{bmatrix} \Psi^{-1}.
\end{align*}
Thus,
\[
Y_j=\Psi
\begin{bmatrix}
X_j &\\
    &X_j^T
\end{bmatrix} \Psi^{-1}
\]
is a square root of $W$. Since $X_j$ is a function of $J_R$, $Y_j$
is a function of $\diag\left(J_R,J_R^T\right)$. This proves the
first part of the theorem.

The second part follows from the second part of Theorem
\ref{Theorem02} and the fact that $\diag\left(\widetilde{X}_j,
{\widehat{X}_j}^T\right)$ is a square root of
$\diag\left(J_R,J_R^T\right)$.
\end{proof}

It is easy to verify that if $W$ is a skew-Hamiltonian matrix, then $W^2$ is also skew-Hamiltonian.
This implies that any function of $W$, which is a polynomial by definition, is a skew-Hamiltonian matrix.
Thus, all the square roots of $W$ which are functions of $W$ are skew-Hamiltonian matrices. The following result
refers to the existence of real square roots of a skew-Hamiltonian matrix.

\begin{corollary}
\label{coro3} Under the assumptions of  {\rm Theorem \ref{Theorem2}},
the following statements hold:

\begin{enumerate}
\item[{(1)}] if $W$ has a real negative eigenvalue,
then $W$ has no real skew-Hamiltonian square roots;

\item[{(2)}] if $W$ has no real negative eigenvalues, then $W$ has precisely $2^{c+{r}}$
real skew-Hamiltonian square roots which are functions of $W$, given by {\rm(\ref{cshamilsqrt})} with the choices of
$j_1, \ldots, j_p$ corresponding to real square roots $F_{n_1}^{(j_1)}, \ldots,
F_{n_{p}}^{(j_{p})}$.

\end{enumerate}

\end{corollary}

It is clear from Theorem \ref{Theorem03} that $W$ may have real negative eigenvalues and yet still have
a real square root; however, the square root will not be a function of $W$.

In \cite[Theorem 2]{FasMMX99} it is shown that

\begin{lemma}
\label{lemma4} Every real skew-Hamiltonian matrix $W$
has a real Hamiltonian square root.
\end{lemma}

The proof is constructive and the key step is based in Lemma \ref{lemma3} -
we can bring $W$ into a real skew-Hamiltonian Jordan form (\ref{skewJordan})
via a sympletic similarity. Further, it is shown that every skew-Hamiltonian matrix
$W$ has infinitely many real Hamiltonian square roots.

The following theorem gives the structure of those real Hamiltonian
square roots.

\begin{theorem}
\label{theorem4} Let $W\in\matR^{2n\times 2n}$ be a nonsingular skew-Hamiltonian matrix and assume the conditions in
{\rm Theorem \ref{Theorem2}.}
\begin{enumerate}
\item[{(1)}] If $W$ has
no real negative eigenvalues, then $W$ has real Hamiltonian square roots which are not functions of $W$ and they form
$2^{p+q}$ parameterized families given by
\begin{equation}  \label{fcsshamilsqrt}
Y_j(\Theta)=\Psi\Theta \diag\left(X_j, {-X_j}^T\right)\Theta^{-1}\Psi^{-1},\quad j=1,\ldots,2^{p+q},
\end{equation}
where $X_j$ denotes a real square root of $J_R$ and $\Theta$ is an
arbitrary nonsingular symplectic matrix which commutes with
$\diag\left(J_R, J_R^T\right)$.
\item[{(2)}] If $W$ has some real negative eigenvalues, then $W$ has real Hamiltonian square roots which are not functions of $W$ and they form
$2^{p+q}$ parameterized families given by
\begin{equation}  \label{fcsshamilsqrt2}
Y_j(\Theta)=\Psi\Theta
\begin{bmatrix}
\widehat{X}_j & K_j\\
 \widehat{K}_j&-\widehat{X}_j^T
\end{bmatrix}
\Theta^{-1}\Psi^{-1},\quad j=1,\ldots,2^{p+q},
\end{equation}
where $\widehat{X}_j$ is a square root for the Jordan blocks of $J_R$ which are not associated with real negative eigenvalues,
$K_j$ and $\widehat{K}_j$ are symmetric block diagonal matrices corresponding to the square roots of the real negative eigenvalues, and
$\Theta$ is an arbitrary nonsingular symplectic matrix which commutes with
$\diag\left(J_R, J_R^T\right)$.
\end{enumerate}
\end{theorem}

\begin{proof}
Equation (\ref{fcsshamilsqrt}) is a special case of Equation
$(\ref{fcshamilsqrt})$ in Theorem \ref{Theorem2}. If $X_j$ is a real
square root of $J_R$ then $\diag\left(X_j,-X_j^T\right)$ is a
Hamiltonian square root of $\diag\left(J_R,J_R^T\right)$ and
Hamiltonian structure is preserved under symplectic similarity
transformations. There are $2^{p+q}$ real square roots of $J_R$
which may be or not functions of $J_R$.

For the second part, assume that $J_R$ in (\ref{rjordanm}),
\begin{equation*}
J_R={\begin{bmatrix}
C_{n_1}(a_1, b_1) &  &  &  &  &  \\
& \ddots &  &  &  &  \\
&  & C_{n_p}(a_p,b_p) &  &  &  \\
&  &  & J_{n_{p+1}}(\lambda_{p+1}) &  &  \\
&  &  &  & \ddots &  \\
&  &  &  &  & J_{n_{p+q}}(\lambda_{p+q})%
\end{bmatrix}},
\end{equation*}
has only one real negative eigenvalue, say $\lambda_k<0$, $k>p$ corresponding to the real Jordan block $J_{n_k}$.

Let $\pm iM_{n_k}$ with $M_{n_k}\in \matR^{n\times n}$ be the two pure imaginary square roots of $J_{n_k}$,
which are upper triangular Toeplitz matrices. See Corollary \ref{coro} and (\ref{JordanBlockSqr}).
Observe that $(\pm iM_{n_k})^2=-M_{n_k}^2=J_{n_k}$.
We will first construct a square root of $\diag\left(J_{n_k},J_{n_k}^T\right)$ which is real and Hamiltonian.
Let $P_{n_k}$ be the reversal matrix of order $n_k$
which satisfies $P_{n_k}^2=I$
(the anti-diagonal entries are all 1's, the only nonzero entries).
The matrices $P_{n_k}M_{n_k}$ and $M_{n_k}P_{n_k}$ are real symmetric
and we have
\[
\begin{bmatrix}
& M_{n_k}P_{n_k}\\
-P_{n_k}M_{n_k}&
\end{bmatrix}^2=
\begin{bmatrix}
-M_{n_k}^2& \\
&-P_{n_k}M_{n_k}^2P_{n_k}
\end{bmatrix}=
\begin{bmatrix}
J_{n_k} \\
& J_{n_k}^T
\end{bmatrix}.
\]
Thus,
\[
\begin{bmatrix}
& M_{n_k}P_{n_k}\\
-P_{n_k}M_{n_k}&
\end{bmatrix} \qquad ( \text{ and also }
\begin{bmatrix}
& -M_{n_k}P_{n_k}\\
P_{n_k}M_{n_k}&
\end{bmatrix})
\]
is a real Hamiltonian square root of
$\diag\left(J_{n_k},J_{n_k}^T\right)$ which is not a function
of
\break $\diag\left(J_{n_k},J_{n_k}^T\right)$.

If $X_1=\diag(F_{n_1},\ldots,F_{n_{p}})$ is a real square root of $\diag(C_{n_1},\ldots,C_{n_{p}})$,\break
$X_2=\diag(L_{n_{p+1}},\ldots,L_{n_{k-1}})$ is a real square root of $\diag(J_{n_{p+1}},\ldots,J_{n_{k-1}})$
and $X_3=\diag(L_{n_{k+1}},\ldots,L_{n_{p+q}})$ is a real square root of $\diag(J_{n_{k+1}},\ldots,J_{n_{p+q}})$, then
\[
\begin{pmat}[{...|..}]
X_1& \cr
    &X_2 \cr
    &      & O & &      & & M_{n_k}P_{n_k}\cr
    &      & & X_3 \cr\-
    &      & & &  -X_1^T \cr
    &      & & &  & -X_2^T \cr
    &      & -P_{n_k}M_{n_k}& &      & & O \cr
    &      & & &      & & & -X_3^T\cr
\end{pmat}
=:
\begin{bmatrix}
\widehat{X}_j & K_j\\
 \widehat{K}_j&-\widehat{X}_j^T
\end{bmatrix}
\]
is a real Hamiltonian square root of $\diag\left(J_R, J_R^T\right)$.
Notice that there are $2^{p+q}$ different square roots with this form.
Thus, for an arbitrary nonsingular symplectic matrix
$\Theta$ which commutes with
$\diag\left(J_R, J_R^T\right)$,
\begin{equation*}
Y_j(\Theta)=\Psi\Theta
\begin{bmatrix}
\widehat{X}_j & K_j\\
 \widehat{K}_j&-\widehat{X}_j^T
\end{bmatrix}
\Theta^{-1}\Psi^{-1},\quad j=1,\ldots,2^{p+q},
\end{equation*}
is a Hamiltonian square root of $W$.

If $W$ has more than one real negative eigenvalue, the generalization is straightforward.
\end{proof}

\section{Algorithms for computing square roots of a skew-Hamiltonian matrix}
\label{AlgorithmSkew}
In this section we will present a structure-exploiting Schur method to compute a real skew-Hamiltonian or a real Hamiltonian square root of a real skew-Hamiltonian matrix $W\in \matR^{2n\times 2n}$ when $W$ does not have real negative eigenvalues.

\subsection{Skew-Hamiltonian square roots}
First we obtain the $PVL$ decomposition of $W \in \matR^{n\times n}$ described in section \ref{HamiltSkew},
\[
U^TWU=\begin{bmatrix}
W_1 & W_2 \\
O & W_1^T
\end{bmatrix},\quad W_2^T=-W_2,
\]
where $U$ is symplectic-orthogonal and $W_1$ is upper Hessenberg.
The matrix $U$ is constructed as a product of elementary symplectic-orthogonal matrices. These are the $2n\times 2n$ Givens rotations matrices of the type
\[
\begin{bmatrix}
I_{j-1}\\
&\cos \theta &    & \sin \theta \\
&            & I_{n-1}\\
            &-\sin \theta & & \cos\theta\\
            &             & &           &I_{n-j}
\end{bmatrix}, \quad 1 \leq j \leq n,
\]
for some angle $\theta \in [-\pi/2,\pi/2[$, and the direct sum of two identical $n\times n$ Householder
matrices
\[
H_j\oplus H_j (\boldsymbol{\upsilon},\beta)=
\begin{bmatrix}
I_n -\beta \boldsymbol{\upsilon} \boldsymbol{\upsilon}^T\\
& I_n -\beta \boldsymbol{\upsilon} \boldsymbol{\upsilon}^T
\end{bmatrix},
\]
where $\boldsymbol{\upsilon}$ is a vector of length $n$ with its first $j-1$ elements equal to zero.
A simple combination of these transformations
can be used to zero out entries in $W$ to accomplish the PVL form.
See Algorithm 1 and Algorithm 5 in \cite[pp. 4,10]{Benner}.
The product of the transformations used in the reductions is accumulated to form the matrix $U$.

Then the standard QR algorithm is applied to $W_1$ producing
an orthogonal matrix $Q$ and a quasi-upper triangular matrix $N_1$ in real Schur form (\ref{RealSchur}) so that
\[
W_1=QN_1Q^T,
\]
and we attain the real skew-Hamiltonian Schur decomposition of $W$,
\begin{equation*}
{\cal{T}}={\cal U}^TW{\cal U}=\begin{bmatrix}
N_1 & N_2 \\
O & N_1^T
\end{bmatrix}, \quad N_2=-N_2^T,
\end{equation*}
where
$
{\cal U}=U \begin{bmatrix} Q & O\\
O & Q
\end{bmatrix}
$
and $N_2=Q^TW_2Q$.

This procedure takes only approximately a $20\%$ of the computational cost the standard $QR$ algorithm
would require to compute the unstructured real Schur decompositon of $W$ \cite[p. 10]{Benner}.

Let
$$
Z=\begin{bmatrix}
X & Y \\
  & X^T
\end{bmatrix}, \qquad Y=-Y^T.
$$
be a skew-Hamiltonian square root of $\cal{T}$. We can solve the equation $Z^2=\cal{T}$ exploiting the structure.
From
$$
\begin{bmatrix}
X & Y \\
  & X^T
\end{bmatrix}
\cdot
\begin{bmatrix}
X & Y \\
  & X^T
\end{bmatrix}=
\begin{bmatrix}
X^2 & XY+YX^T\\
0      & \left(X^T\right)^2
\end{bmatrix}=
\begin{bmatrix}
N_{1} & N_{2}\\
0      & N_{1}^T
\end{bmatrix},
$$
we have
\begin{equation}
\label{Alg2Eq1}
X^2=N_{1}
\end{equation}
and
\begin{equation}
\label{Alg2Eq2}
XY+YX^T=N_{2}.
\end{equation}
Equation (\ref{Alg2Eq1}) can be solved using Higham's real Schur method
(see Algorithm 1, Section \ref{SectionRealShur}, page \pageref{pageRealSchur}) and
it is not difficult to show that $X$ inherits $N_1$'s quasi-upper triangular structure.
Equation (\ref{Alg2Eq2}) is a Lyapunov equation which can be solved efficiently since $X$
is already in quasi-upper triangular real Schur form and $Y$ is skew-symmetric.
The techniques are the same as for the Sylvester equation. See \cite[Chapter 16]{Bartels,Highbook0}.

If the partitions of $X=(X_{ij})$, $Y=(Y_{ij})$ and $N_{2}=(N_{ij})$ are conformal with $N_{1}$ block structure,
\begin{align*}
X&=\begin{bmatrix}
X_{11} & X_{12} & \ldots & X_{1m}\\
       & X_{22} & \ldots & X_{2m}\\
       &        & \ddots & \vdots\\
       &        &        & X_{mm}
\end{bmatrix}, \qquad
Y=\begin{bmatrix}
Y_{11}    & -Y_{21}^T    & \ldots  & -Y_{m1}^T\\
Y_{21} & Y_{22}    & \ldots  & -Y_{m2}^T\\
\vdots    &  \vdots   & \ddots  & \vdots\\
Y_{m1} & Y_{m2} & \ldots  & Y_{mm}
\end{bmatrix}, \quad \\\\
N_2&=\begin{bmatrix}
N_{11}    & -N_{21}^T    & \ldots  & N_{m1}^T\\
N_{21}  & N_{22}    & \ldots  & -N_{m2}^T\\
\vdots    &  \vdots   & \ddots  & \vdots\\
N_{m1}  & N_{m2} & \ldots  & N_{mm}
\end{bmatrix}, \qquad
\begin{array}{rr}
Y_{ii}=-Y_{ii}^T, \quad N_{ii}=-N_{ii}^T, \\
i=i,\ldots,m.
\end{array}
\end{align*}
then, from (\ref{Alg2Eq2}), we have
\begin{equation}
\sum_{k=i}^mX_{ik}Y_{kj}+\sum_{k=j}^mY_{ik}X_{jk}^T=N_{ij}
\end{equation}
and
\begin{align*}
X_{ii}Y_{ij}+Y_{ij}X_{jj}^T&=N_{ij}-\sum_{k=i+1}^mX_{ik}Y_{kj}-\sum_{k=j+1}^mY_{ik}X_{jk}^T\\
&=N_{ij}-\sum_{k=i+1}^mX_{ik}Y_{kj}-\sum_{k=j+1}^mY_{ik}X_{jk}^T.\\
\end{align*}
These equations may be solved successively for $Y_{mm}, Y_{m,m-1}, \ldots, Y_{m1}$,\break
$Y_{m-1,m-1}, Y_{m-1,m-2}, \ldots, Y_{m-1,1}$, \ldots, $Y_{22}, Y_{21}$ and $Y_{11}$.
We have to solve
\begin{align}
\nonumber
&X_{ii}Y_{ij}+Y_{ij}X_{jj}^T=N_{ij}-\sum_{k=i+1}^mX_{ik}Y_{kj}-\sum_{k=j+1}^iY_{ik}X_{jk}^T+\sum_{k=i+1}^mY_{ki}^TX_{jk}^T,\\
\label{SylvEq}
&\qquad  i=m,m-1,\ldots,1\\
\nonumber
&\qquad  j=i,i-1,\ldots,1.
\end{align}
Since $X_{ii}$ are of order 1 or 2, each system (\ref{SylvEq}) is a linear system of order 1,2 or 4 and
is usually solved by Gaussian elimination with complete pivoting.
The solution is unique because $X_{ii}$ and $-X_{jj}^T$ have no eigenvalues in common.
See Section \ref{SectionExistence}.



\textbf{Algorithm 2 [Skew-Hamiltonian real Schur method]}

\medskip

\begin{enumerate}

\item[1.] compute a real skew-Hamiltonian Schur decomposition of $W$,
$$
{\cal T}={\cal U}^TW{\cal U}=\begin{bmatrix}
N_{1} & N_{2}\\
0      & N_{1}^T
\end{bmatrix};
$$

\medskip

\item[2.] use Algorithm 1 to compute a square root $X$ of $N_{1}$, $X^2=N_{1}$;

\medskip

\item[3.] solve the Sylvester equation $XY+YX^T=N_{2}$ using (\ref{SylvEq}) and form
$$
Z=\begin{bmatrix}
X & Y \\
  & X^T
\end{bmatrix};
$$

\smallskip

\item[4.] obtain the skew-Hamiltonian square root of $N$, ${\cal X}={\cal{U}}Z{\cal{U}}^T$.

\end{enumerate}

The cost of the real skew-Hamiltonian Schur method for $W\in \matR^{2n\times 2n}$ is measured in flops as follows.
The real skew-Hamiltonian Schur factorization of $W$ costs about $3(2n)^3$ flops \cite{GolVL96,Benner}.
The computation of $X$ requires $n^3/6$ flops, the computation of the skew-symmetric solution $Y$ requires
about $n^3$ flops \cite[p. 368]{GolVL96}
and the formation of ${\cal X}={\cal{U}}Z{\cal{U}}^T$ requires $3(2n)^3/2$ flops.
The total cost is approximately $5(2n)^3$ flops. Comparing with
the overall cost of Algorithm 1, the unstructured real Schur method, which is about $17\times (2n)^3$ flops, Algorithm 2
requires considerably less floating point operations.

\subsection{Hamiltonian square roots}
Analogously, let {$Z$ be a Hamiltonian square root} of $\cal{T}$,
$$
{Z=\begin{bmatrix}
X & Y \\
  & -X^T
\end{bmatrix}, \qquad Y=Y^T}
$$
(which is not a function of $\cal{T}$). To solve the equation {$Z^2=\cal{T}$}, observe that,
from
$$
{\begin{bmatrix}
X & Y \\
  & -X^T
\end{bmatrix}
\cdot
\begin{bmatrix}
X & Y \\
  & -X^T
\end{bmatrix}=
\begin{bmatrix}
X^2 & XY-YX^T\\
0      & \left(X^T\right)^2
\end{bmatrix}=
\begin{bmatrix}
N_{1} & N_{2}\\
0      & N_{1}^T
\end{bmatrix}}
$$
it follows
\begin{equation}
\label{Alg3Equ1}
{X^2=N_{1}}
\end{equation}
and
\begin{equation}
\label{Alg3Equ2}
{XY-YX^T=N_{2}}.
\end{equation}
Equation (\ref{Alg3Equ1}) can be solved using Higham's real Schur method and
Equation (\ref{Alg3Equ2}) is a {singular Sylvester equation} with infinitely many symmetric solutions.
See \cite[Proposition 7]{FasMMX99}. Again, the structure can be exploited and we have to solve
\begin{align}
\nonumber
&X_{ii}Y_{ij}-Y_{ij}X_{jj}^T=N_{ij}-\sum_{k=i+1}^mX_{ik}Y_{kj}+\sum_{k=j+1}^iY_{ik}X_{jk}^T-\sum_{k=i+1}^mY_{ki}^TX_{jk}^T,\\
\label{SylvEq2}
&\qquad  i=m,m-1,\ldots,1\\
\nonumber
&\qquad  j=i,i-1,\ldots,1.
\end{align}
The solution of the linear system (\ref{SylvEq2}) may not be unique but it always exists.


\bigskip

{\textbf{Algorithm 3 [Hamiltonian real Schur method]}}

\medskip

\begin{enumerate}

\item[1.] compute a real skew-Hamiltonian Schur decomposition of $W$,
$$
{{\cal T}={\cal U}^TW{\cal U}=\begin{bmatrix}
N_{1} & N_{2}\\
0      & N_{1}^T
\end{bmatrix};}
$$

\medskip

\item[2.] use Algorithm 1 to compute a square root $X$ of $N_{1}$, {$X^2=N_{1}$;}

\medskip

\item[3.] obtain one solution for the Sylvester equation {$XY-YX^T=N_{2}$} using (\ref{SylvEq2})
and form
$$
{Z=\begin{bmatrix}
X & Y \\
  & -X^T
\end{bmatrix};}
$$
\smallskip

\item[4.] obtain the {Hamiltonian square root} of $W$, {${\cal X}={\cal U}Z{\cal{U}}^T$.}

\end{enumerate}

\section{Numerical examples} \label{sectionExamples}

We implemented Algorithms 2 and 3 in \textsc{Matlab} 7.5.0342 (R2007b) and
used the Matrix Function Toolbox by Nick Higham available in Matlab
Central website
\texttt{{http://www.mathworks.com/matlabcentral}}. To find the square
root $X$ in step 2 we used the function \textsf{sqrtm\_real} of this toolbox
and to solve the linear systems (\ref{SylvEq}) in step 3 of Algorithm 2 we
used the function \textsf{sylvsol} (the solution is always unique). In step
3 of Algorithm 3 the linear systems (\ref{SylvEq2}) are solved using
Matlab's function \textsf{pinv} which produces the solution with the
smallest norm when the system has infinitely many solutions.

\smallskip

Let {$\bar{\mathcal{X}}$ be an approximation} to a square root of
$W$ and define the \textsl{residual}
\begin{equation*}
{E=\bar{\mathcal{X}}^{2}-W.}
\end{equation*}%
Then, we have ${\bar{\mathcal{X}}^{2}=W}$ $+{E}$ and, as observed by Higham
\cite[p. 418]{Hig97}, the stability of an algorithm for computing a square
root {$\bar{\mathcal{X}}$ }$\ $of ${W}$ corresponds to the residual $E$
being small relative to ${W}$. Furthermore, for {$\bar{\mathcal{X}}$ }$\ $%
computed with \textsf{sqrtm\_real,} Higham gives the following error bound

\begin{equation*}
\frac{\Vert {E}\Vert _{F}}{\Vert W\Vert _{F}}\leq \left( 1+cn\frac{\Vert
\bar{\mathcal{X}}\Vert _{F}^{2}}{\Vert W\Vert _{F}}\right) u
\end{equation*}%
where ${\Vert \cdot \Vert }_{F}$ is the Frobenius norm, $c$ is a constant of
order 1, $n$ is the dimension of $W$ and $u$ is the roundoff unit.
Therefore, the real Schur method is stable provided that the number%
\begin{equation*}
\alpha ({\mathcal{X})=}\frac{\Vert \bar{\mathcal{X}}\Vert _{F}^{2}}{\Vert
W\Vert _{F}}
\end{equation*}
is small.

We expect our structure-preserving algorithms, Algorithm 2 (skew-Hamiltonian
square root) and Algorithm 3 (Hamiltonian square root) to be as accurate as
Algorithm 1 (real Schur method) which ignores the structure. The numerical
examples that follow illustrate that the three algorithms are all quite
accurate when $\alpha ({\mathcal{X})}$ is small.

\begin{example}
The skew-Hamiltonian matrix
\begin{align*}
{W=%
\begin{bmatrix}
\boldsymbol{e}\boldsymbol{e}^T & A \\
-A^T & \boldsymbol{e}\boldsymbol{e}^T%
\end{bmatrix}%
, \qquad A=%
\begin{bmatrix}
0 & 10^{-6} & 1 & 0 & 0 \\
-10^{-6} & 0 & 1 & 10^{-6} & 0 \\
-1 & -1 & 0 & 10^{-6} & 1 \\
0 & -10^{-6} & -10^{-6} & 0 & 1 \\
0 & 0 & {-1} & {-1} & 0%
\end{bmatrix}%
,}
\end{align*}
where $\boldsymbol{e}$ is the vector of all ones, has one complex conjugate
eigenvalue pair and 3 positive real eigenvalues (all with multiplicity 2).

The relative residuals of both the skew-Hamiltonian and Hamiltonian square
roots computed with {\rm Algorithm 2} and {\rm Algorithm 3} are $4\times 10^{-15}$, the
same as for the square root delivered by {\rm Algorithm 1.}
\end{example}


\begin{example}
The eigenvalues of the skew-Hamiltonian matrix
\begin{align*}
W=%
\begin{bmatrix}
A & B \\
B & A^T%
\end{bmatrix}%
, \qquad & A=%
\begin{bmatrix}
0 & -10^{-6} & 0 & 0 \\
10^{-6} & 0 & 0 & 0 \\
0 & 0 & 0 & 10^{-6} \\
0 & 0 & -10^{-6} & 0%
\end{bmatrix}%
, \\
& B=%
\begin{bmatrix}
0 & 1 & 2 & 3 \\
-1 & 0 & 2 & 3 \\
-2 & -2 & 0 & 3 \\
-3 & -3 & -3 & 0%
\end{bmatrix}%
,
\end{align*}
are all very close to pure imaginary (four distint eigenvalues).

The relative residuals of the square roots delivered by all the three
methods are {$4\times 10^{-16}$.}
\end{example}

If $W$ has negative real eigenvalues there are no real square roots which
are functions of $W$. However, all these algorithms can be applied and
complex square roots will be obtained. In step 2 of Algorithms 2 and 3 a
complex square root is computed and so we get a complex skew-Hamiltonian and
a complex Hamiltonian square-root.

\begin{example}

For random matrices $A$, $B$ and $C$ (values  drawn from a uniform
distribution on the unit interval), the computed square roots of the
skew-hamiltonian matrix  of order $2n=50$ (several cases)
\begin{equation*}
{W=%
\begin{bmatrix}
A & B-B^T \\
C-C^T & A^T%
\end{bmatrix}%
}
\end{equation*}
also have relative residuals of order at most $10^{-14}$.
\end{example}


\section{Conclusions}

Based in the real skew-Hamiltonian Jordan form, we gave a clear
characterization of the square roots of a real skew-Hamiltonian matrix $W$.
This includes those square roots which are functions of $W$ and those which
are not. Although the Jordan canonical form is the main theoretical tool in
our analysis, it is not suitable for numerical computation. We have designed
a method for the computation of square roots of such structured matrices.
An important component of our method is
the real Schur decomposition tailored for skew-Hamiltonian matrices, which has been
used by others in solving problems different from ours.

Our algorithm requires considerably less floating point operations (about
70\% less) than the general real Schur method due to Higham. Furthermore, in numerical experiments, our algorithm
has produced results which are as accurate as those obtained with \textsf{sqrtm\_real}.


\end{document}